\documentclass[12pt,  amscd, amsfonts,leqno]{amsart}

      \theoremstyle{plain}
      \newtheorem{theorem}{Theorem}[section]
      \newtheorem{lemma}[theorem]{Lemma}
      
      \newtheorem{proposition}[theorem]{Proposition}

      \theoremstyle{definition}
      
  \newtheorem{example}[theorem]{Example}

      \theoremstyle{remark}

\begin{document}

\title{Dynamical Estimates on a Class of Quadratic Polynomial Automorphisms of $\mathbb C^3$ }
\author{Ozcan Yazici}
\thanks{The author was partially supported by the Qatar National Research Fund, NPRP project 7-511-1-098}
\date{}
\address{Texas A\&M University at Qatar }
\email{ozcan.yazici@qatar.tamu.edu}

\pagestyle{myheadings}
\begin{abstract} 
Quadratic automorphisms of $\mathbb C^3$ are classified up to affine conjugacy into seven classes by Forn\ae ss and Wu. Five of them contain irregular maps with interesting dynamics. In this paper, we focus on the maps in the fifth class and make some dynamical estimates for these maps. 
\end{abstract}

\maketitle

\section{Introduction}

A polynomial automorphism $f$ of $\mathbb C^n$ and its inverse $f^{-1}$  define meromorphic maps of $\mathbb P^n$ which are well-defined away from the indeterminacy set $I^+$ (resp. $I^-$). An automorphism  $f$ of $\mathbb C^n$ is called \textit{regular} if the indeterminacy sets $I^+$ and $I^-$ of $f$ and $f^{-1}$ are disjoint. This assumption on the  indeterminacy sets gives a relation between the degree of $f$ and the degree of $f^{-1}$ which allows the construction of an invariant measure.  An extensive study of regular maps can be found in \cite {S}.

The more general class of \textit{weakly regular} automorphisms of $\mathbb C^n$ was studied by Guedj and Sibony in \cite{GS}. Roughly speaking, these are the maps for which $I^{+}$ and $X^+=f(\{t=0\}\setminus I^+)$ are disjoint where $\{t=0\}$ is the hyperplane at infinity in $\mathbb P^n.$

In \cite{FW}, the quadratic polynomial automorphisms of $\mathbb C^3$ are classified up to  affine conjugacy into $7$ classes, $5$ of which are dynamically interesting as they are non-linear. In \cite{CF}, Coman and Forn\ae ss studied these classes by estimating the rates of escape of orbits to infinity and by describing the subsets of $\mathbb C^3$ where such orbits occur. Using these estimates, they construct invariant measures for the maps in some of the classes.  The maps in $2$ of  these classes are the most complicated. They have the form
$$H_4(x,y,z)=(P(x,y)+az,Q(y)+x,y),$$
$$ H_5(x,y,z) = (P(x,y)+az, Q(x)+by,x),$$
where $\max\{deg(P ), deg(Q) \}=2$ and $a\neq 0\neq b.$ The complication is due to their strange behavior at infinity. They both map $\{t=0\}\setminus I^+$ onto $I^-$. For $H_5$, the extended indeterminacy set $I^+_{\infty} $ consists of union of two lines, $\{t=x=0\}\cup \{t=y=0\}$ and there are two  points, $\{[1:0:0:0], [0:1:0:0]\}=I^+_{\infty}\cap I^-$, where the orbits that tend to infinity with a slow rate may accumulate.

The dynamics of $H_4$ is explained in detail by Coman in \cite{C}. We focus on the class $H_5$ and work on the maps $$H:(x,y,z)\to (xy+az,x^2+by,x),\; ab\neq0, $$  
for simplicity. The iterates of $H$  will be denoted by $H^n(w)=w_n=(x_n,y_n,z_n)$ where $w=(x,y,z)\in \mathbb C^3.$ The induced map $H:\mathbb P^3\to \mathbb P^3$ is defined by 
$$H[x:y:z:t]=[xy+azt:x^2+byt:xt:t^2].$$  
Then $H$ is not (weakly) regular since $I^+=\{t=x=0\}$, $X^+=I^-=\{t=z=0\}$ and $X^+\cap I^+=[0:1:0:0].$ 

The paper is organized as follows. In Section \ref{dynamicsofH-1}, we define the invariant sets $K^-$ and $U^-$ for $H^{-1}$. \cite[Lemma 6.1]{CF} implies that the orbits  $H^{-n}(w)$ of points in $U^-$  escape to $[0:0:1:0]$ with a super-exponential rate $(const)^{3^n}$. It was known by \cite[Lemma 6.2]{CF} that the orbits $H^{-n}(w)$ of points in $K^-$ have at most exponential growth. In Proposition \ref{jd}, we show that if the coefficient  $b$ of $H$ has modulus $|b|>1$, then the orbits $H^{-n}(w)$ are bounded for all $w\in K^-.$

In Section \ref{dynamicsofH},  the $H$-invariant sets $U^+$ and $K^+=\mathbb C^3\setminus U^+$ are defined. It was known by \cite[Proposition 6.4]{CF} that on $U^+$, the orbits $H^n(w)$ escape to infinity at the super-exponential rate  $(const)^{2^n},$ while on $K^+$ the orbits escape to infinity at a much slower rate. We give shorter proof of these facts for our simplified map in Theorem \ref{gb} and Lemma \ref{2}. Points of $K^+$ accumulates  on two lines at infinity, namely $I^+=\{t=x=0\}$ and $\{t=y=0 \}$. This implies that unbounded $H$-orbits of points in $K^+$ may accumulate at two different points $[1:0:0:0]\cup[0:1:0:0]$ (see Theorem \ref{orbitK}). This fact causes a complicated behavior of the automorphism at infinity.

We  construct an $H^{-1}$ invariant current  $\sigma$ of bidimension $(1,1)$ on $\mathbb P^3$ which is supported on $\partial {\overline {K^-}}$ (see Theorem \ref{inv}). This may allow us to construct a non trivial $H^{-1}$-invariant measure in $\mathbb C^3$ by wedging with the Green's current $T^+$. We note that it is impossible to construct an invariant measure by using the Green's currents only as in the case of regular maps since $T^+\wedge T^+=T^-\wedge T^-=0$ in $\mathbb C^3$ (see Section \ref{invm}).
 
Theorem \ref{gb} $(iv)$ shows that  the orbits $H^n(w)$ of points in $K^+$ may escape to infinity with at most the super-exponential growth rate $(const)^{(\sqrt 3)^n}.$ However, it is not clear that if there is any orbit of a point in $K^+$ which has this sharp growth rate. In the case of $a=0$, $H$ becomes a map of $\mathbb C^2$ and there are some lines  in $K^+\subset \mathbb C^2$ on which the orbits have the sharp growth rate of $(const)^{(\sqrt 3)^n}$ (see Section \ref{2dimmodel}).

\noindent \textbf{Acknowledgments.} This paper is part of the author's Ph.D. thesis and he is grateful to his advisor Prof. Dan Coman for his guidance and support.

\section {Dynamics of $H^{-1}$}\label{dynamicsofH-1}
$H^{-1}$ is given by $$H^{-1}(x,y,z)=\left (z,\frac{y-z^2}{b} ,\frac{1}{a}\left (x-z\frac{y-z^2}{b}\right )\right).$$ We will denote the iterates of $H^{-1}$ by $H^{-n}(w)=(x_n,y_n,z_n)$, where $w=(x,y,z)\in \mathbb C^3$. Let us define the sets 

\begin{eqnarray*} V^{+}&=&\{w\in \mathbb C^{3}: |z|>\max\{R,|x|, (1+\delta)|y|^{1/2}\}\}, \nonumber \\
U^-&=&\cup_{n=0}^{\infty}H^{n}(V^+), \\
K^-&=&\mathbb C^3 \setminus U^-,\nonumber
\end{eqnarray*} 
where $\delta>0$ and $R$ is sufficiently big. In \cite[Lemma 6.1, Lemma 6.2]{CF}, it was shown that the $H^{-1}$-orbits of points in $U^-$ converge locally uniformly on $U^-$ to $[0:0:1:0]$ with super-exponential rate $(const)^{3^n}$, and for  $w\in K^-$, $||H^{-n}(w)||\leq C^n M(w)$ where $C>1$ and $M(w)>0$  depends on $w$ continuously.   
We will show that when $|b|>1$, the $H^{-1}$-orbits of points in $K^-$ are actually bounded. Let's assume that $|b|=1+\delta$, $\delta>0$. 
\begin{proposition}\label{jd} $(i)$ There exists  $R_0> 1$ such that for all $R>R_0$, we have $H^{-1}(V^+)\subset V^+$ and if $w=(x,y,z)\in V^+$ then
$$ C_1|z|^3<|z_1|<C_2 |z|^3 $$
 where $C_1$ and $C_2$ depend on  $\delta$.

$(ii)$ $||H^{-n}(w)||$ is bounded when $w\in K^{-}.$
\end{proposition}

\begin{proof} Let $\alpha>0$ be a constant satisfying $\alpha+\frac{1}{(1+\delta)^2}<1$ and $R_0=\left(\frac{|b|}{\alpha}\right)^{\frac{1}{2}}$. First we note that on $V^+$, 

\begin{eqnarray}\label{med} |x|<|z|,\; |y|<\frac{|z|^2}{(1+\delta)^2} \;\text{and}\; |z|>R>R_0=\left(\frac{|b|}{\alpha}\right)^{\frac{1}{2}}.
\end{eqnarray}
 
It follows from  (\ref{med}) that on $V^+$ we have 
\begin{eqnarray*}\left|z_1-\frac{z^3}{ab}\right|&=&\left|\frac{x}{a}-\frac{zy}{ab}\right |<\frac{|z|}{|a|}+\frac{|z|^3}{(1+\delta)^2|ab|}\\&<&\left(\frac{|b|}{|z|^2}+\frac{1}{(1+\delta)^2} \right)\frac{|z|^3}{|ab|}  \leq\left(\alpha+\frac{1}{(1+\delta)^2} \right)\frac{|z|^3}{|ab|},
\end{eqnarray*}
which implies the estimate in part $(i)$. 

On $V^+$, $R<|z|$, $|x_1|=|z|$ and 
\begin{eqnarray*}(1+\delta)|y_1|^{1/2}&\leq& \frac{(1+\delta)}{|b|^{1/2}}|y-z^2|^{1/2}\leq \frac{(1+\delta)}{|b|^{1/2}}\left(\frac{|z|^2}{(1+\delta)^2}+|z|^2  \right)^{1/2}\\&=&\frac{(1+\delta)}{|b|^{1/2}}\left(1+\frac{1}{(1+\delta)^2} \right)^{1/2}|z|.
\end{eqnarray*}
Since $|z_1|>C_1|z|^3$, these estimates imply that $H^{-1}(V^+)\subset V^+$ when $R$ is big enough.

We will show part (ii) now. By definition, on $K^{-}$ we have that   $$|z_n|\leq \max \{R,|x_n|,(1+\delta)|y_n|^{\frac{1}{2}}\}=:M_n$$ for all $n\geq0$ and   $|x_{n+1}|=|z_n|\leq M_n$.

Let us consider the case when $|z_n|\leq 1.$ Then 
\begin{eqnarray*}(1+\delta)|y_{n+1}|^{\frac{1}{2}}\leq (1+\delta)\left(\frac{|y_n|}{|b|}+\frac{1}{|b|}\right)^\frac{1}{2}&<&(1+\delta)\left(\frac{M_n^2}{(1+\delta)^2|b|}+\frac{1}{|b|}\right)^\frac{1}{2}\\ &=& \left(\frac{M_n^2}{|b|}+\frac{(1+\delta)^2}{|b|}    \right)^\frac{1}{2}      \\ &\leq&  M_n\left(\frac{1}{|b|}+\frac{(1+\delta)^2}{|b|R^2}\right)^\frac{1}{2}<M_n . \end{eqnarray*} 
When $|z_n|>1$ 
\begin{eqnarray*}(1+\delta)|y_{n+1}|^{\frac{1}{2}}= (1+\delta)\left|\frac{x_n}{z_n}-\frac{az_{n+1}}{z_n}\right |^\frac{1}{2}  &<&(1+\delta)\left(M_n+|a|M_{n+1}\right)^\frac{1}{2}. \end{eqnarray*} 
We also have that $|x_{n+1}|=|z_n|\leq M_n$. Hence 
 $$M_{n+1}\leq\max\{M_n,(1+\delta)(M_n+|a|M_{n+1})^{1/2}\}.$$ 
 If the right hand side of the above inequality is equal to $M_n$ then $M_{n+1}\leq M_n$ and we are done. Hence we can assume that $M^2_{n+1}\leq (1+\delta)^2(M_n+|a|M_{n+1}).$   Choosing $R$ big enough we obtain that   $$M_{n+1}<\frac{M_{n+1}^2-(1+\delta)^2|a|M_{n+1}}{(1+\delta)^2}<M_n.$$ Therefore $||H^n(w)||$ is bounded for $w \in K^-.$ 
\end{proof}

We now construct the Green's function of $H^{-1}.$ Note that $\deg H^{-n}=3^n$.   We let 
\begin{eqnarray*} G_n(w)=\frac{1}{3^n}\log^+||H^{-n}(w)|| \;\text{and}\; \tilde G_n(w)=\frac{1}{3^n}\log^+|z_n|.
\end{eqnarray*}
 The estimates in \cite[Lemma 6.1, Lemma 6.2]{CF} imply that  $G_n$ and $\tilde G_n$ converge locally uniformly to the same Green's function $G^{-}$. Hence the Green's function $G^-$ is pluriharmonic on $U^-$, $K^-=\{G^-=0\}$ and $G^-\circ H^{-1}=3G^-$. We define the Green's current by $\mu^-=dd^cG^-$. Then $H^*\mu^-=\frac{1}{3}\mu^-$ and supp $ \mu^-=\partial K^-$. 

\section{Dynamics of $H$} \label{dynamicsofH} 
 
 We will denote the $nth$ iteration of $H$ by $H^n(w)=w_n=(x_n,y_n,z_n).$
For $\epsilon>0$ and $R$ big enough, we define the sets
\begin{eqnarray}\label{K+defn}
V^-&=&\left\{(x,y,z)\in \mathbb C^3:|xy|>\max\left\{R,2|az|,|x|^{3/2},\frac{1}{\epsilon}|y|^{3/2}\right\}\right\}, \nonumber \\
U^+&=&\cup_{n=0}^{\infty}H^{-n}(V^-), \\
K^+&=&\mathbb C^3 \setminus U^+.\nonumber
\end{eqnarray}

In  \cite{CF},  they showed that the orbits $H^n(w)$ of points in $U^+$ escape to infinity with super-exponential  growth rate $(const)^{2^n}.$ For the sake of completeness, we give a short proof of this fact for our simplified map.

\begin{theorem} \label{gb}There exists $\epsilon>0$ and $R_0>1/ \epsilon^{10}$ such that for all $R>R_0$, we have $H(V^-)\subset V^-$ and

\begin{eqnarray} \label{6.3}
\frac{|xy|}{2}&<&|x_1|<\frac{3|xy|}{2} \\ (1-|b|\epsilon^2)|x|^2 &<& |y_1|< (1+|b|\epsilon^2) |x|^2. \nonumber
\end{eqnarray}
Hence on $U^+$, $H^n(w)$ escape to infinity with the super-exponential growth $(const)^{2^n}$.
%$(ii)$ There is a positive continuous function $M$ on $\mathbb C^3$ such that for all $w\in K^+$ and $n\geq 0$, we have $$\max\{|x_n|, |y_n|, |z_n|\}\leq M(w)^{(\sqrt {3})^n}. $$
\end{theorem}
\begin{proof} For the points in $V^-$, we have that $|2az|<|xy|$. Hence $|x_1-xy|=|az|< |xy|/2$ which proves the first inequality above.  On $V^-$, $|xy|>\frac{1}{\epsilon}|y|^{\frac{3}{2}}$ which implies that $|y|<\epsilon^2|x|^2.$ Thus $$|y_1-x^2|=|by|\leq|b|\epsilon^2|x|^2$$ and this implies the second inequality above. We now prove that $V^-$ is invariant under $H$. On $V^-$, $|xy|>|x|^{\frac{3}{2}}$ which implies that $|x|<|y|^2$. Since  $\epsilon^2|x|^3>|xy|>R>1/\epsilon^{10}$, 
\begin{eqnarray}\label{xybig}|x|>\frac{1}{\epsilon^4}\; \text{and}\; |y|>\frac{1}{\epsilon^2}
\end{eqnarray}
 on $V^-$. Using this with the first inequality in (\ref{6.3}), we obtain that  
 \begin{eqnarray*}|x_1y_1|>|x|^3|y|\frac{1-|b|\epsilon^2}{2}&=& 2|a||z_1||x|^2|y|\frac{1-|b|\epsilon^2}{4|a|}\\&\geq& 2|a||z_1|\frac{1-|b|\epsilon^2}{\epsilon^{10}4|a|}>2|az_1|.
 \end{eqnarray*}
 
 (\ref{6.3}) and (\ref{xybig}) with the inequality $|y|<\epsilon^2|x|^2$ imply that  
\begin{eqnarray*}\max\left\{|x_1|^\frac{3}{2},\frac{1}{\epsilon}|y_1|^{\frac{3}{2}}\right\}&\leq&
 \max\left\{\left(\frac{3|xy|}{2}\right)^{\frac{3}{2}}, \frac{1}{\epsilon}\left(1+|b|\epsilon^2\right)^{\frac{3}{2}} |x|^3 \right\}\\& \leq& \max\left\{ \epsilon \left(\frac{3}{2}\right)^{3/2} |x|^{5/2} |y|,  \frac{1}{\epsilon} (1+|b|\epsilon^2)^{3/2} |x|^3\right \} \\&\leq& |x|^3|y|\max\left\{\epsilon^3 \left(\frac{3}{2}\right)^{3/2},  \epsilon\left(1+|b|\epsilon^2\right)^{\frac{3}{2}}\right \}\leq|x|^3|y| \frac{1-|b|\epsilon^2}{2}<|x_1y_1|. 
 \end{eqnarray*}
Thus $H(V^-)\subset V^-$. 
\end{proof}
We now discuss the dynamics of $H$  on $K^+$. Let 
\begin{eqnarray}\label{M_n} M_n=M_n(w):=\max\left\{R,2|az_n|,|x_n|^{3/2},\frac{1}{\epsilon}|y_n|^{3/2}\right\}.
\end{eqnarray}

\begin{lemma}\label{2}On $K^+$,  if $R$ is sufficiently large, then we have  \\
(i) $|x_n|\leq\frac{3M_{n-1}}{2}$,\\
(ii) $M_n\leq\max\{(\frac{3M_{n-1}}{2})^{\frac{3}{2}},\frac{1}{\epsilon}|y_n|^{\frac{3}{2}}\}$,\\
(iii) $M_n\leq C\max\{M^{\frac{3}{2}}_{n-1},|x_{n-1}|^3\}$ for some constant $C$,\\
(iv) $M_n\leq \tilde{M}(w)^{(\sqrt 3)^n}$ for some continuous function  $\tilde{M}(w)$. Hence $||H^n(w)||\leq C(w)^{(\sqrt 3)^n}$ where $C(w)$ is a continuous function of $w$. 
\end{lemma}
\begin{proof} 

(i) Note that on $K^+$, $|x_ny_n|\leq M_n$ for all $n\geq 0$. It follows from the definition of map $H$ and the set $M_n$ that $$|x_n|\leq|x_{n-1}y_{n-1}|+|az_{n-1}|\leq M_{n-1}+ \frac{M_{n-1}}{2}=\frac{3M_{n-1}}{2}.$$

(ii) The inequality $|z_n|=|x_{n-1}|\leq M_{n-1}^{\frac{2}{3}}$ with the estimate in $(i)$ implies that   
\begin{eqnarray*}
M_n&\leq&\max\left\{M_{n-1},2|a|M_{n-1}^{\frac{2}{3}},\left(\frac{3M_{n-1}}{2}\right)^{\frac{3}{2}},\frac{1}{\epsilon}|y_n|^{\frac{3}{2}}\right\}\\&=&\max\left\{ \left(\frac{3M_{n-1}}{2}\right)^{\frac{3}{2}},\frac{1}{\epsilon}|y_n|^{\frac{3}{2}}\right\}.
\end{eqnarray*}

(iii) By definition of $y_n$ \begin{eqnarray*} |y_n|^{\frac{3}{2}}&\leq&(|x_{n-1}|^2+|b||y_{n-1}|)^{\frac{3}{2}}\\ &\leq& 2^{\frac{3}{2}}\max\left\{|x_{n-1}|^3,|b|^\frac{3}{2}|y_{n-1}|^\frac{3}{2} \right\}.\end{eqnarray*} Then $$\frac{1}{\epsilon}|y_n|^\frac{3}{2}\leq2^{\frac{3}{2}}\max\left\{\frac{1}{\epsilon}|x_{n-1}|^3,|b|^\frac{3}{2}|M_{n-1}|\right \}.$$ This estimate with (ii) proves (iii).

(iv) Since $H$ is a degree two polynomial map, $M_{n-1}\leq C M_{n-2}^2$ for some constant $C$ . So by (i) and (iii) we obtain that 
$$M_n\leq C \max\left\{M_{n-2}^3, \left(\frac{3M_{n-2}}{2}\right)^3\right\}=\tilde CM_{n-2}^3.$$
Hence if $n$ is even  then $M_n\leq \tilde CM_{n-2}^3\leq ...\leq (\tilde C M_0)^{(\sqrt{3})^n}. $ If $n$ is odd then $M_n\leq (\tilde CM_1)^{3^{\frac{n-1}{2}}}\leq ((\tilde CM_1)^2)^{(\sqrt{3})^n}$.  Thus $M_n\leq \tilde{M}(w)^{(\sqrt 3)^n}$ where 
$\tilde{M}(w)=\max\{ \tilde C M_0,  (\tilde CM_1)^2\}.$

\end{proof}

We now define the Green's function $G^+$ of $H$ by $$G^+(w)=\lim_{n\to\infty}\frac{1}{2^n}\log^+||H^n(w)||=\lim_{n\to\infty}\frac{1}{2^n}\log^+|x_n|=\lim_{n\to\infty}\frac{1}{2^n}\log^+|y_n|. $$ 

\begin{theorem}\cite{CF} The above limits exist and are equal. $G^+$ is a continuous psh function in $\mathbb C^3$ and it is pluriharmonic on $U^+$. Moreover, $K^+=\{G^+=0\}$ and $G^+\circ H=2G^+$.  
\end{theorem}
 The Green's current of $H$ is defined by $\mu^+=dd^c G^+.$ Then $H^*\mu^+=2\mu^+$ and  supp $\mu^+=\partial K^+$.  Let us consider the  induced map $H$ on $\mathbb P^3$ and the Fubini-Study form $\omega$  on $\mathbb P^3$. For a more general class of automorphisms, Sibony (\cite[Theorem 1.6.1]{S}) showed that $\frac{1}{2^n}(H^n)^*\omega$ converges to a   closed positive current $T_+$ of bidegree $(1,1)$ which satisfies $H^* T_+=2T_+$ on $\mathbb P^3$.  Moreover, by \cite[Theorem 1.8.1]{S}, $T_+$ does not charge the hyperplane at infinity and  $T_+|_{\mathbb C^3}=\mu^+$ has mass one in $\mathbb C^3$.  

Unlike the regular automorphisms (see \cite {S} for regular automorphisms), orbits of points in $K^+$ may escape to infinity. For example, $H(0,y,0)=(0,b^ny,0)\to[0:1:0:0]$ if $|b|>1$. By Lemma \ref{2} (iv), any such orbit may escape to infinity with a  smaller super-exponential rate $(const)^{(\sqrt 3)^n}$.  We will show that unbounded orbits of points in $K^+$ may accumulate only at two points, $P=[0:1:0:0]$ and $Q=[1:0:0:0].$ First we show that points in $K^+$ accumulate at infinity on the set $I^+_{\infty}=I^+\cup\{t=y=0\}.$ 

\begin{theorem}\label{K+}$\overline{K^+}=K^+\cup I_{\infty}$
\end{theorem}

\begin{proof}Since $$\overline{{V^-}^{c}}=\left\{[x:y:z:t]\in \mathbb P^3:|xy|\leq\max\left\{ R|t^2|,2|azt|,|x|^{\frac{3}{2}}|t|^{\frac{1}{2}},\frac{1}{\epsilon}|y|^{\frac{3}{2}}|t|^{\frac{1}{2}}   \right\} \right\}$$ we have that  $$\overline{{V^-}^{c}}={V^-}^c\cup\{x=t=0\}\cup\{y=t=0 \}={V^-}^c\cup I_{\infty}.$$ 
Then $$\overline{K^+}\subset \overline{{V^-}^{c}}\subset {V^-}^c\cup I^+_{\infty}$$ which implies that $\overline{K^+}\subset K^+\cup I^+_{\infty}. $

Let $\tilde H$ be a homogeneous representation of the extension of  $H$ to $\mathbb P^3$ so that $||\tilde H(w)||\leq ||w||^2$ and $T_+$ be the Green's current of $H$ in $\mathbb P^3.$  By \cite[Theorem 1.6.1]{S}, we have on $\mathbb C^4$, $$\frac{1}{2^n}\log ||\tilde H^n(w)||\searrow \tilde G(w)\; \text{and}\; \pi^*T_+=dd^c\tilde G,$$ where $\pi:\mathbb C^4\setminus \{0\}\to \mathbb P^3$ is the canonical map.  Note that for any  $p\in I^+_{\infty}$, there exists a $\tilde p\in \mathbb C^4\setminus \{0\}$ such that  $\tilde H^2(\tilde p)=0$ and $p=\pi(\tilde p)$. Since $\tilde G(w)\leq \frac{1}{4}\log||\tilde H^2(w)||$, by comparison theorem for Lelong numbers (see \cite {D1}) we have that  $$\nu (\tilde G,\tilde p)\geq \nu\left(\frac{1}{4}\log||\tilde H^2||,\tilde p\right)>0.$$ Hence $\tilde p\in \;\text{supp}\; \pi^* T_+\subset \pi^{-1}(\text{supp}\; T_+)$, that is, $p\in \text{supp}\; T_+.$  \cite[Theorem 6.5]{CF} implies that $\text{supp}\; T_+\subset\partial  \overline {K^+}$, hence $I^+_{\infty}\subset \partial  \overline {K^+}.$
\end{proof}

We should note that the above result is true for weakly regular maps by \cite[Theorem 2.2]{GS} with $I^+_{\infty}$ replaced by $I^+$.  So Theorem \ref {K+} is of interest since the similar result holds for maps which are not weakly regular.

\begin{theorem} \label{orbitK}Unbounded orbits of points in $K^+$ under $H$ can only cluster at $I^+_{\infty}\cap I^-=\{[1:0:0:0],[0:1:0:0]\}.$
\end{theorem}

\begin{proof} Since $K^+$ is invariant under $H$, Theorem \ref{K+}  implies that an unbounded orbit can only cluster on $I^+_{\infty}$. Let $w\in K^+$ and $w_{n_i}=H^{n_i}(w)\rightarrow w_0.$ If $w_0\notin I^- $, then $w_{n_i-1}= H^{n_i-1}(w)=H^{-1}(w_{n_{i}})\rightarrow H^{-1}(w_0)= X^-.$ Since $H^{-1}$ is weakly regular $H^{n_i-1}(w)$ avoids a neighborhood of $I^{-}=I^-_{\infty}.$ Thus $H^{-(n_i-1)}$ is well-defined. Since  $H^{-1}$ is weakly regular $X^{-}$ is $H^{-1}$ attracting. So $ H^{-(n_i-1)}H^{n_i-1}(w)=w\rightarrow X^-$. This contradicts the fact that $w $ is a fixed point in $\mathbb C^k.$ Thus $w_0\in I^+_{\infty}\cap I^-=\{[1:0:0:0],[0:1:0:0]\}.$
\end{proof}

\section{Invariant Measures}\label{invm}
For regular automorphisms of $\mathbb C^n$, Sibony (\cite[Theorem 2.5.2]{S}) constructed  an invariant probability measure $\mu={T_+}^l\wedge {T_-}^{n-l} $    where $\dim\; I^-=l-1$. It is impossible to construct an invariant measure for $H$ and $H^{-1}$  by using the powers of the Green's currents since $\mu^+\wedge \mu^+=\mu^-\wedge \mu^-=0$ in $\mathbb C^3$.  Indeed, $$\mu_{\epsilon}=dd^c\max\{G^+, \epsilon\}\wedge \mu^+\to \mu^+\wedge\mu^+$$ as $\epsilon \to 0$.  We note that $K^+=\{G^+=0\}$ and supp $\mu_{\epsilon}\subset \partial K^+.$ Let $w\in \partial K^+$ and $B$ be a neighborhood of $w$ in $\mathbb C^3$ such that $G^+<\epsilon$ in $B$. Hence $\max\{G^+,\epsilon\}=\epsilon$ and $\mu_{\epsilon}=0$ on $B$. Thus $\mu_{\epsilon}=0=\mu^+\wedge\mu^+$ in $\mathbb C^3.$ 

When $f^{-1}$ is weakly regular and $I^-$ is $f$-attracting, by \cite[Theorem 3.1]{GS}, there is an invariant current $\sigma_{s}$ of bidimension $(s,s)$ where the $\dim\; X^-=s-1$. Then the wedge product $\mu:=\sigma_s\wedge T_-^s$ is  an $f$-invariant measure. It is well-defined  since $G^-$ is locally bounded  near $\overline {K^+}.$ 

In our case, $H^{-1}$ is weakly regular. If $|b|>1$, then $I^{-}$ is $H-$attracting and the construction of invariant measure as in \cite{GS} works for $H$. However, when $|b|\leq1$, $I^-$ is not $H-$attracting and  their construction does not work. 

Since $H^{-1}$ is weakly regular, $X^{-}=[0:0:1:0]$ is an attracting point for $H^{-1}$. Using this fact with the estimates on $K^{-}$ and some ideas from \cite{GS}, we construct an $H^{-1}-$ invariant current of bidimension $(1,1)$ which is supported on $\partial\overline {K^-}.$ 

\begin{theorem}\label{inv}There is a  closed positive current $\sigma$ of bidimension $(1,1)$ on $\mathbb P^3$ such that $(H^{-1})^*\sigma=2\sigma$ and supp $\sigma \subset \partial \overline {K^{-}} $.
\end{theorem}

\begin{proof}Let $\omega$ be the standard  K\"{a}hler  form in $\mathbb P^3$ and $\omega|_{\mathbb C^3} $ be the restriction of $\omega$  to $\mathbb C^3.$    We define $$R_N=\frac {1}{N}\sum_{n=1}^{N}\frac{(H^{-n})^*\omega^2|_{\mathbb C^3}}{2^n}.$$ We still denote by $R_N$ the trivial extension to $\mathbb P^3.$ Then 
\begin{eqnarray*}||R_{N}||&=&\int_{\mathbb P^3} \frac{1}{N}\sum_{n=1}^{N}\frac{(H^{-n})^*\omega^2}{2^n}\wedge \omega=\int_{\mathbb C^3} \frac{1}{N}\sum_{n=1}^{N}\frac{(H^{-n})^*\omega^2}{2^n}\wedge \omega \\ &=&\int_{\mathbb C^3} \frac{1}{N}\sum_{n=1}^{N}\frac{(H^{n})^*\omega}{2^n}\wedge \omega^2 =\int_{\mathbb P^3} \frac{1}{N}\sum_{n=1}^{N}\frac{(H^{n})^*\omega}{2^n}\wedge \omega^2=1.\end{eqnarray*}
Therefore there is a subsequence $R_{N_{j}}$ which converges to  a current $\sigma$ in the sense of currents.  Since $||R_{N_{j}}||=1$, $\sigma$ has mass $1$ in $\mathbb P^{3}$ and it is invariant under the pullback by $H|_{\mathbb C^3}^{-1}$. Indeed,
\begin{eqnarray*}
(H^{-1})^*R_{N_j}&=&\frac{2}{N_{j}}\sum_{n=1}^{N_{j}} \frac{(H^{-(n+1)})^*\omega^2}{2^{n+1}}\\&=&\frac{2}{N_j}\left(N_jR_{N_j}-\frac{(H^{-1})^*\omega^2}{2}+\frac{(H^{-(N_j+1)})^*\omega^2}{2^{N_j+1}}\right) \rightarrow 2\sigma,
\end{eqnarray*}
as $ ||\frac{(H^{-(N_j+1)})^*\omega^2}{2^{N_j+1}}||=1.$
On the other hand $(H^{-1})^*$ is continuous on currents in $\mathbb C^3$. Thus $(H^{-1})^*\sigma=2\sigma$ on $\mathbb C^3.$ 

We first prove that $\text{supp}\; \sigma\subset \overline{K^{-}}$. By \cite[Theorem 2.2]{GS}, $$\overline{K^{-}}=K^{-} \cup I^{-}.$$  

Let $X^{-}=\cap_{j=1}^{\infty}U_{j}$ where $U_j$'s are decreasing  open sets in $\mathbb P^3$ and $\epsilon>0.$ Since dim$(X^-)=0$ and $T_+$ is a current of bidimension $(2,2)$, $T_+\wedge \omega^2(X^-)=0$.   Hence there is a $U_{J}$ such that $T_+\wedge \omega^2(\overline{U_J})<\epsilon.$ Let $B\subset \mathbb P^3\setminus \overline{K^{-}}$ be a ball. Since $H^{-1}$ is weakly regular, $X^-$ is attracting, with basin $U^-$ in $\mathbb C^3$.  Thus there exists $M>0$ such that $H^{-n}(B)\subset U_J$ for all $n\geq M.$ By \cite[Theorem 1.6.1]{S}, $$\frac{(H^n)^*(\omega)}{2^n}\rightarrow T_+\;\text{and}\;   \frac{1}{N}\sum_{n=1}^{N}\frac{(H^n)^*(\omega)}{2^n} \rightarrow T_+.$$ Also the subsequence 
\begin{eqnarray*} T_{N_{j}}:&=&\frac{1}{N_j}\sum_{n=M}^{N_j} \frac{(H^n)^*(\omega)}{2^n}\\ &=& \frac{1}{N_j}\left(\sum_{n=1}^{N_j}\frac{(H^n)^*(\omega)}{2^n}- \sum_{n=1}^{M-1}\frac{(H^n)^*(\omega)}{2^n}\right)\rightarrow T_+ 
\end{eqnarray*}

as $N_j\rightarrow \infty.$ Hence  $T_{N_j}\wedge \omega^2\rightarrow T_+\wedge \omega^2$  as measures and $$\limsup_{N_j\to\infty} T_{N_j}\wedge \omega ^2(\overline{U_J})\leq T_+\wedge \omega^2(\overline{U_J})<\epsilon. $$ This implies that $  T_{N_j}\wedge \omega ^2(\overline{U_J})\leq \epsilon$ for all $N_j>N$ for some $N>0.$ Since $R_{N_j}$ does not charge the hyperplane at infinity, we have that
\begin{eqnarray*}
\int_BR_{N_j}\wedge \omega&=&\int_{B\cap \mathbb C^3}\frac{1}{N_j}\sum_{n=1}^{N_j}\frac{(H^{-n})^*\omega^2}{2^n}\wedge \omega \\ &=&  \frac{1}{N_j}\sum_{n=1}^{N_j}\int_{H^{-n}(B\cap \mathbb C^3)}\frac{(H^n)^*(\omega)}{2^n}\wedge \omega^2 \\ &\leq&  \frac{1}{N_j}\left( \sum_{n=1}^{M-1}\int_{H^{-n}(B\cap\mathbb C^3)}\frac{(H^n)^*(\omega)}{2^n}\wedge \omega^2  + \sum_{n=M}^{N_j}\int_{U_J}\frac{(H^n)^*(\omega)}{2^n}\wedge \omega^2    \right)\\  &\leq& \frac{M-1}{N_j}+\int_{\overline{U_J}}T_{N_j}\wedge\omega^2\leq 2\epsilon,
\end{eqnarray*}
if $N_j$ is big enough.
This shows that $\sigma\wedge\omega(B)=0.$ 

Now we will show that $\sigma$ has no mass in the interior of $\overline {K^{-}}$. Let $U\subset\subset \text{int}\;\overline{K^{-}}$. Since $\overline{K^{-}}=K^-\cup I^-_{\infty}$, $U$ is actually contained in $K^-.$ By \cite[Lemma 6.3]{CF}, there is $C>1$ such that $||H^{-n}(z)||\leq C^n $ for all $z\in U$ and $n>0$.  In $\mathbb C^3$,
\begin{eqnarray*} R_{N_j}=\frac{1}{N_j}\sum_{n=1}^{N_j}\left(\frac{(H^{-n})^*\omega}{(\sqrt{2})^n}\right)^2&\leq&\left(\frac{1}{N_j^{\frac{1}{2}}} \sum_{n=1}^{N_j}\frac{(H^{-n})^*\omega}{(\sqrt{2})^n}\right)^2 \\ &=& (dd^cG_{N_j})^2 
\end{eqnarray*}   
where $G_{N_j}(z):=\frac{1}{N_j^{\frac{1}{2}}}\sum_{n=1}^{N_j}\frac{\log(1+||H^{-n}(z)||^2)^{\frac{1}{2}}}{(\sqrt 2)^{n}}$. \\ On $U$, $$0\leq G_{N_j}\leq \frac{1}{N_j^{1/2}}\sum_{n=1}^{N_j}\frac{n\log C}{(\sqrt 2)^n}. $$ Thus $G_{N_j}$ converges to $0$ locally uniformly on  int $K^-$ and  hence $$R_{N_j}\leq (dd^cG_{N_j})^2\rightarrow 0,$$ which implies that $\sigma$ has no mass on int $K^-$. Thus $\text{supp}\;\sigma\subset  \partial \overline {K^{-}}. $

\end{proof}

\section{A Two Dimensional Model}\label{2dimmodel}

We will consider the case $a=0$. Then  $H$ becomes a map of $\mathbb C^2$, 
$$H(x,y)=(xy, x^2+by).$$
We note that $H$ is not an automorphism anymore. It determines a map $H:\mathbb P^2\to\mathbb P^2$ by $H([x:y:t])=[xy:x^2+byt:t^2]$. The indeterminacy point is $I=[0:1:0]$.  The sets $K^+$ and $U^+$ are defined as in (\ref{K+defn}) without the $z$ coordinate.  Then all the estimates for the map in $\mathbb C^3$ in Sections \ref{dynamicsofH}  hold for the reduced map in $\mathbb C^2$.

In the following examples, when $b^4=1$, we have some lines which are contained in $K^+$ and invariant under the second iterate $H^2$ of $H$. If $w$ lies on these lines then $H^n(w)\approx C(w)^{(\sqrt 3)^n}$, so the maximal growth on $K^+$ given by the estimates from Section \ref{dynamicsofH} does occur.

\begin{example} The second iterate of $H$ is $$H^2(x,y)=H(xy,x^2+by)=(xy(x^2+by), x^2(y^2+b)+b^2y),$$ so $x_2=xy(x^2+by)$ and $y_2=x^2(y^2+b)+b^2y.$ 

If $b^4=1$ and  $y^2=-b$ then $y_2^2=b^4y^2=-b^5=-b$. We have four choices for $b$.\\
Case 1. If $b=1$ then $H^2(x,y)=(x^3y+xy^2,x^2(y^2+1)+y)$. Hence  $H^2(x,i)=(ix^3-x,i)$ and $H^2(x,-i)=(-ix^3-x,-i).$ So the lines $\{y=i\}$ and $\{y=-i\}$ are invariant under $H^2$.\\
Case 2. If $b=-1$ then $H^2(x,y)=(x^3y-xy^2,x^2(y^2-1)+y)$. Hence $H^2(x,1)=(x^3-x,1)$ and $H^2(x,-1)=(-x^3-x,-1).$ So the lines $\{y=1\}$ and $\{y=-1\}$ are invariant under $H^2$.\\
Case 3. If $b=i$ then $H^2(x,y)=(x^3y+ixy^2,x^2(y^2+i)-y)$. Hence $H^2(x,e^{\frac{3\pi i}{4}})=(e^{\frac{3\pi i}{4}}x^3+x,-e^{\frac{3\pi i}{4}})$  and $H^2(x,-e^{\frac{3\pi i}{4}})=(-e^{\frac{3\pi i}{4}}x^3+x,e^{\frac{3\pi i}{4}}).$\\
Case 4. If $b=-i$ then  $H^2(x,y)=(x^3y-ixy^2,x^2(y^2-i)-y)$. Hence $H^2(x,e^{\frac{\pi i}{4}})=(e^{\frac{\pi i}{4}}x^3+x,-e^{\frac{\pi i}{4}})$  and $H^2(x,-e^{\frac{\pi i}{4}})=(-e^{\frac{\pi i}{4}}x^3+x,e^{\frac{\pi i}{4}}).$

In all of the cases above, since $b^4=1$ and $y^2=-b$, we have $$H^2(x,\sqrt{-b})=(x^3\sqrt{-b}-b^2x,b^2\sqrt{-b}).$$ Thus $|y_{2n}|=1$, $|x_{2n}|\approx |x|^{3^n}$, $|x_{2n+1}|=|x_{2n}y_{2n}|\approx |x|^{3^n}$ and $|y_{2n+1}|=|x_{2n}^2+by_{2n}|\approx |x|^{2\cdot 3^n}$ for all $n\geq 0$. Hence $H^n(w)\approx C(w)^{(\sqrt 3)^n}$ if $w$ is contained in these lines. 
\end{example}

 \end{document}